\documentclass[12pt]{amsproc}
\newtheorem{theorem}{\sc Theorem}[section]
\newtheorem{lemma}[theorem]{\sc Lemma}
\newtheorem{proposition}[theorem]{\sc Proposition}
\newtheorem{corollary}[theorem]{\sc Corollary}

\newcommand{\al }{\alpha }

\begin{document}

\title[Coprime automorphisms of finite groups]{Coprime automorphisms of finite groups}
\author{Cristina Acciarri}

\address{Cristina Acciarri: Dipartimento di Scienze Fisiche, Informatiche e Matematiche, Universit\`a degli Studi di Modena e Reggio Emilia, Via Campi 213/b, I-41125 Modena, Italy  and\newline
Department of Mathematics, University of Brasilia, Brasilia-DF, 70910-900 Brazil}
\email{cristina.acciarri@unimore.it}

\author{Robert M. Guralnick}

\address{Robert M. Guralnick:  Department of Mathematics, University of Southern California, Los Angeles, CA90089-2532, USA}
\email{guralnic@usc.edu}

\author{Pavel Shumyatsky }
\address{ Pavel Shumyatsky: Department of Mathematics, University of Brasilia,
Brasilia-DF, 70910-900 Brazil}
\email{pavel@unb.br}
\thanks{The first and the third authors were supported by  the Conselho Nacional de Desenvolvimento Cient\'{\i}fico e Tecnol\'ogico (CNPq),  and Funda\c c\~ao de Apoio \`a Pesquisa do Distrito Federal (FAPDF), Brazil. The second author was partially supported by a Simons
Foundation fellowship and NSF grant DMS-1901595.}

\begin{abstract} 
Let $G$ be a finite group admitting a coprime automorphism $\al$ of order $e$. Denote by $I_G(\al)$ the set of commutators $g^{-1}g^\al$, where $g\in G$, and by $[G,\al]$ the subgroup generated by $I_G(\al)$. We study the impact of $I_G(\al)$ on the structure of $[G,\al]$. Suppose that each subgroup generated by a subset of $I_G(\al)$ can be generated by at most $r$ elements. We show that the rank of $[G,\al]$ is $(e,r)$-bounded. Along the way, we establish several results of independent interest. In particular, we prove that if every element of $I_G(\al)$ has odd order, then $[G,\al]$ has odd order too. Further, if every pair of elements from $I_G(\al)$ generates a soluble, or nilpotent, subgroup, then $[G,\al]$ is soluble, or respectively nilpotent.
\end{abstract}

\maketitle

\section{Introduction}

An automorphism $\al$ of a finite group $G$ is coprime if $(|G|,|\al|)=1$. We denote by $C_G(\al)$ the fixed-point subgroup $\{x\in G; \ x^\al=x\}$ and by $I_G(\al)$ the set of all commutators $g^{-1}g^\al$, where $g\in G$. Then $[G,\al]$ stands for the subgroup generated by $I_G(\al)$. 

It is well known that properties of the centralizer $C_G(\al)$ of a coprime automorphism have strong influence over the structure of $G$. There is a wealth of results illustrating this phenomenon, probably the most famous of which is Thompson's Theorem that if $\al$ has prime order and $C_G(\al)=1$, then $G$ is nilpotent \cite{tho}. Over the years, this was generalized in several directions. In particular, Khukhro proved that if $G$ admits an automorphism $\al$ of prime order $p$ with $C_G(\al)$ of order $m$, then $G$ has a nilpotent subgroup of $(m,p)$-bounded index and $p$-bounded class \cite{khu1}. Throughout, we use the term $(a,b,c\dots)$-bounded to mean ``bounded from above by some function depending only on the parameters $a,b,c\dots$".  Further, if $G$  admits a coprime automorphism $\al$ of prime order $p$ with $C_G(\al)$ of rank $r$, then $G$ has characteristic subgroups $R$ and $N$ such that $N/R$ is nilpotent of $p$-bounded class, while $R$ and $G/N$ have $(p,r)$-bounded ranks \cite{khu2}. Recall that the rank of a finite group $G$ is the least number $r$ such that each subgroup of $G$ can be generated by at most $r$ elements.

Given a coprime automorphism $\al$ of a finite group $G$, there is a kind of (rather vague) duality between $C_G(\al)$ and $I_G(\al)$. Note that since $|G|=|C_G(\al)||I_G(\al)|$, if one of $C_G(\al), I_G(\al)$ is large then the other is small.  Our purpose in the present article is to show that also properties of $I_G(\al)$ may strongly impact the structure of $G$. It is easy to see that if $|I_G(\al)|\leq m$, then the order of $[G,\al]$ is $m$-bounded. Indeed, since $|I_G(\alpha)|\leq m$, the index of the centralizer $[G:C_G(\alpha)]$ is at most $m$ and we can choose a  normal subgroup $N\leq C_G(\alpha)$ such that $[G:N]\leq m!$. Observe that $[G,\alpha]$ commutes with $N$ (see item (iii) of Lemma \ref{20}) and therefore the centre of $[G,\alpha]$ has index at most $m!$. The Schur theorem \cite[Theorem 4.12]{rob} now tells us that $[G,\alpha]'$ has $m$-bounded order. We can pass to the qutient $G/[G,\alpha]'$ and, without loss of generality, assume that $[G,\alpha]$ is abelian. But then $[G,\alpha]=I_G(\alpha)$ and so $|[G,\alpha]|\leq m$.
We address the question whether a rank condition imposed on the set $I_G(\al)$ has an impact on the structure of $G$. We emphasize that $I_G(\al)$ in general is not a subgroup and therefore the usual concept of rank does not apply to $I_G(\al)$. Instead we consider the condition that each subgroup of $G$ generated by a subset of $I_G(\al)$ can be generated by at most $r$ elements. Our main result is as follows.
\begin{theorem} \label{main}
Let $G$ be a finite group admitting a coprime automorphism $\al$ of order $e$ and suppose that any subgroup generated by a subset of $I_G(\al)$ can be generated by $r$ elements. Then $[G,\al]$ has $(e,r)$-bounded rank.
\end{theorem}

The much easier particular case of the theorem for $e=2$ was earlier dealt with in \cite{as}. 
As might be expected, Theorem \ref{main} depends on the classification of finite simple groups. Along the way, we establish several results of independent interest. In particular, we prove the following result.

\begin{theorem} \label{main2} Let $G$ be a finite group admitting a coprime automorphism $\al$ such that $g^{-1}g^\al$ has odd order for every $g\in G$. Then $[G,\al]\leq O(G)$.
\end{theorem}

As usual, $O(G)$ stands for the maximal normal subgroup of odd order of $G$. Recall that an immediate corollary of Glauberman's celebrated $Z^*$-theorem is that if $G$ contains an involution $x$ such that $[g,x]$ has odd order for every $g\in G$, then $[G,x]\leq O(G)$ \cite{glau}. A theorem obtained in \cite{guro} states that if $G$ contains an element $x$ of prime order $p$ such that $[g,x]$ has $p$-power order for every $g\in G$, then $[G,x]\leq O_p(G)$. Thus, one may wonder whether the assumption that $\al$ is coprime in Theorem \ref{main2} is really necessary. In  Section 3 we give  examples showing that the theorem is no longer true if the assumption is omitted.

It is well known that if any pair of elements of a finite group generates a soluble (respectively nilpotent) subgroup, then the whole group is soluble (respectively nilpotent).  Indeed, Guest \cite{gue} showed that if $g$ is of prime order at least $5$ and if every two conjugates of $g$
generate a soluble group, then $g$ is in the soluble radical of $G$. 
We will establish a similar result for groups with coprime automorphisms.
\begin{theorem} \label{main3} Let $G$ be a finite group admitting a coprime automorphism $\al$. If any pair of elements from $I_G(\al)$ generates a soluble subgroup, then $[G,\al]$ is soluble. If any pair of elements from $I_G(\al)$ generates a nilpotent subgroup, then $[G,\al]$ is nilpotent.
\end{theorem}

Note that in $S_n$ for $n\ge 5$ any four transpositions generate a soluble subgroup. Furthermore there are almost simple groups containing elements of order $3$ such that any three conjugates generate a soluble subgroup (see Guest \cite{gue}).These examples show that the coprimeness assumption is needed in the previous theorem.  It seems likely that  the assumption can be removed in the case where $\al$ is of prime order at least $5$. 

\section{Preliminary results}

All groups considered in this paper are finite. The Feit-Thompson theorem that groups of odd order are soluble \cite{FT} will be used without explicit references. We start with a collection of well-known facts about coprime automorphisms of finite groups (see for example \cite{gore}).

\begin{lemma}\label{20} Let a group $G$ admit a coprime automorphism $\al$. The following conditions hold:
\begin{itemize} 
\item[(i)] $G=[G,\al]C_G(\al)$ and $|I_G(\al)|=[G:C_G(\al)]$;
\item[(ii)] If $N$ is any $\al$-invariant normal subgroup of $G$ we have $C_{G/N}(\al)=C_G(\al)N/N$, and $I_{G/N}(\al)=\{gN \mid g\in I_G(\al)\}$;
\item[(iii)] If N is any $\al$-invariant normal subgroup of $G$ such that $N=C_N(\al)$, then $[G,\al]$ centralizes $N$;
\item[(iv)] The group $G$ possesses an $\al$-invariant Sylow $p$-subgroup for each prime $p\in\pi(G)$.
\end{itemize}
\end{lemma}

Throughout, by a simple group we mean a nonabelian simple group. We will often use without special references the well-known corollary of the classification that if a simple group $G$ admits a coprime automorphism $\alpha$ of order $e$, then $G=L(q)$ is a group of Lie type and $\al$ is a field automorphism. Furthermore, $C_G(\al)=L(q_0)$ is a group of the same Lie type defined over a smaller field such that $q=q_0^e$ (see \cite{GLS3}). 

\begin{lemma}\label{0000} Let $r$ be a positive integer and $G$ a simple group admitting a coprime automorphism $\al$ of order $e > 1$. 
\begin{enumerate}
\item If the order of $[P,\al]$ is at most $r$ whenever $P$ is an $\al$-invariant Sylow subgroup of $G$, then the order of $G$ is $r$-bounded.
\item If the rank of $[P,\al]$ is at most $r$ whenever $P$ is an $\al$-invariant Sylow subgroup of $G$, then the rank of $G$ is $r$-bounded.
\end{enumerate}
\end{lemma}
\begin{proof} We know that $G=L(q_0^e)$ is a group of Lie type and $q_0=p^s$ is a $p$-power for some prime $p$. Moreover $C_G(\al)=L(q_0)$ is a group of the same Lie type. Choose an $\al$-invariant Sylow $p$-subgroup $U$ in $G$ such
that $C_U(\al)$ is a Sylow $p$-subgroup of $C_G(\al)$.   

Comparing the orders of $G$ and $U$ note that $|G| \le |U|^3$ (see \cite{GLS3})  and so for (1),  it suffices to show that $U$ is bounded in terms
of $|[U, \al]|$.   Note that $|U|=q_0^{ed}$ for some $d$ and $|C_U(\al)|= q_0^d$, it follows that $|[U,\al]|\geq q_0^{d(e-1)} > |U|^{1/2}$.  

We now prove (2).   To bound the rank of $G$, it suffices to bound the rank of each Sylow $\ell$-subgroup of $G$ \cite{gu2, Lu}.      If $\ell \ne p$, it
is well known that the rank of a Sylow $\ell$-subgroup of $G$ is at most the (untwisted)  Lie rank of $G$ plus the rank of the Weyl
group \cite[Sec. 4.10]{GLS3}.  

Now consider $U$.  By \cite[3.3.1, Thm. 3.3.3]{GLS3},  there exists an elementary abelian $\al$-invariant subgroup $A$ of $U$ of order $p^{es f(m)}$
where $m$ is the untwisted Lie rank of $G$ and $f(m)$ is some function which grows quadratically in the Lie rank  $m$.  Moreover, this subgroup $A$
is a product of root subgroups if the group is untwisted and a product of abelian subgroups of root subgroups in the twisted case and each
of these subgroups is $\al$-invariant.   
It follows as in the proof of (1), that $[A, \al]$ has rank at least $(e-1)sf(m)$.   Our hypothesis implies that $e, s$ and $m$ are bounded.
Since $|U| \le p^{e s m^2}$, this bounds the rank of $U$ and the Lie rank of $G$.  The result now follows. 
 \end{proof}

\begin{lemma}\label{1111} Let $G$ be a simple group admitting a coprime automorphism $\al$. There is a prime $p\in\pi(G)$ such that $G$ is generated by two $p$-subgroups $P_1$ and $P_2$ with the property that $P_1=[P_1,\al]$ and $P_2=[P_2,\al]$.
\end{lemma} 
\begin{proof}  Again, $G$ is a group of Lie type, say over the field of $q$ elements and $\al$ is a field automorphism of odd order $e$ (and so $e \ge 5$ or $e = 3$ and $G = Sz(q)$). Furthermore, $C_G(\al)$ is the group of the same Lie type over the field of $q_0$ elements where $q = q_0^e$. In particular, $\al$ normalizes a Borel subgroup $B = UT$, where $T$ is a torus and $U$ is a Sylow $p$-subgroup, where $q$ is a $p$-power. Then $B^- = U^-T$ is the opposite Borel subgroup (with $U \cap U^-=1$). This is obtained by conjugating $B$ by the longest element in the Weyl group. We claim that $G = \langle [U,\al], [U^-,\al] \rangle$.

First consider the rank 1 groups:  $ \mathrm{PSL}_2(q),  \mathrm{PSU}_3(q),  {^2\mathrm{G}_2(q)}, \mathrm{Sz}(q)$. Note
that $|[U,\al]| \geq (q/q_0)^m$ where $|U|=q^m$ and by inspection of the maximal subgroups
(cf \cite{Wilson}), deduce that the only maximal subgroup containing $[U,\al]$ is $B$ and so the result holds.  

In a similar way we treat the group ${^2}F_4(2^d)$ -- all maximal subgroups are known (\cite[4.9.3]{Wilson}) and none of them contains both $[U,\al]$ and $[U^-,\al]$.

So assume that $G$ has (twisted) Lie rank at least 2. Since the automorphism $\al$ normalizes any root subgroup, each parabolic subgroup $P$ containing $B$ (and similarly for $B^-$) is  $\al$-invariant. Note that $P = QL$ where $Q$ is the unipotent radical and $L$ is the standard Levi subgroup. Note that $[L,L]$ is a central product of quasisimple groups of Lie type of smaller rank and so, by induction on the rank, we have 
$[L,L] \le \langle [U \cap L,\al], [U^- \cap L,\al] \rangle$. Observe that $[L,L]$ is generated by the root subgroups corresponding to the system of a subset of positive roots. So every simple root is contained in some parabolic subgroup. We conclude that in particular the root subgroups $U_{\pm a}$ are contained in $\langle [U,\al], [U^-,\al] \rangle$ for each positive simple root $a$. Since the positive simple root subgroups $U_a$ generate $U$ (and $U_{-a}$ generate $U^-$), we see that $U$ and $U^-$ are contained in $\langle [U,\al], [U^-,\al] \rangle$ and it is well known (see Section 2.9 in \cite{GLS3}) that these generate $G$.
\end{proof}

Throughout, the term ``semisimple group'' means direct product of simple groups. 

\begin{lemma}\label{01} Let $C$ be a positive integer and $G$ a finite group admitting a coprime automorphism $\al$ such that $G=[G,\al]$. Suppose that the order of $[P,\al]$ is at most $C$ whenever $P$ is an $\al$-invariant Sylow subgroup of $G$. Then the order of $G$ is $C$-bounded.
\end{lemma}
\begin{proof} First,  suppose that $G$ is abelian, in which case  $P=[P,\al]$ and  $p\leq |P|\leq C$. It follows that $|G|\leq C^f$, where $f$ is the number of primes less than or equal to $C$. So we assume that $G$ is nonabelian. If $G$ is simple, then the result is immediate from Lemma \ref{0000}(1). If $G$ is semisimple and $\al$ transitively permutes the simple factors, then any $\al$-invariant Sylow subgroup $Q$ is a product $Q_1\times\dots\times Q_l$, where $\al$ transitively permutes the factors $Q_i$. Observe that $|[Q,\al]|\geq|Q_1|^{l-1}$ and the result follows. So suppose that $G$ has proper $\al$-invariant normal subgroups. Let $\pi(G)=\{p_1,\dots,p_k\}$ and for each $i\leq k$ choose an $\al$-invariant Sylow $p_i$-subgroup $P_i$ in $G$. Let $s(G)$ denote the product $\prod_{1\leq i\leq k}|[P_i,\al]|$. Obviously $s(G)\leq C^k$ and note that $k$ is $C$-bounded. Thus $s(G)$ is $C$-bounded and so we will use induction on $s(G)$. Suppose first that $\al$ acts nontrivially on every $\al$-invariant normal subgroup of $G$. Let $M$ be a minimal $\al$-invariant normal subgroup. By induction the order of $G/M$ is $C$-bounded. The subgroup $M$ is either an elementary abelian $p$-group for some prime $p\leq C$ or a semisimple group. In any case $[M,\al]$ has $C$-bounded order. Since $[M,\al]$ is normal in $M$, which has $C$-bounded index in $G$, we conclude that the normal closure $\langle[M,\al]^G\rangle$ has $C$-bounded order. Because of minimality of $M$ we have $\langle[M,\al]^G\rangle=M$ and so the order of $G$ is $C$-bounded. This completes the proof in the particular case where $\al$ acts nontrivially on every $\al$-invariant normal subgroup of $G$.

Next, suppose that $G$ has nontrivial normal subgroups contained in $C_G(\al)$. Let $N$ be the product of all such subgroups. In view of the above $G/N$ has $C$-bounded order. Since $N\leq Z(G)$, we deduce from Schur's Theorem \cite[Theorem 4.12]{rob} that $G'$ has $C$-bounded order. Hence the result.
\end{proof}

\begin{lemma}\label{02} Let $G=H\langle a\rangle$ be a group with a normal subgroup $H$ and an element $a$ such that $(|H|,|a|)=1$ and $H=[H,a]$. Suppose that $G$ faithfully acts by permutations on a set $\Omega$ in such a way that the element $a$ moves only $m$ points. Then the order of $G$ is $m$-bounded.
\end{lemma}
\begin{proof} First, note that the order of $a$ is obviously $m$-bounded. Another useful observation is that because of Lemma \ref{01} without loss of generality we can assume that $H$ is a $p$-group for some prime $p$. Let $\Omega_0$ be a nontrivial $G$-orbit. If $a$ moves no points in $\Omega_0$, then taking into account that $H=[H,\al]$ we conclude that also $H$ acts trivially on $\Omega_0$, a contradiction. Therefore, $\al$ moves at least $2$ points on every nontrivial $G$-orbit and so there are at most $m/2$ orbits of $G$ in $\Omega$. Since $G$ embeds into a subdirect product $G_1\times \cdots \times G_r$, with $G_i$ transitive on the $i$th nontrivial $G$-orbit, without loss of generality the action of $G$ on $\Omega$ can be assumed transitive and so it is sufficient to bound the cardinality of $\Omega$. Consider the corresponding permutational representation of $G$ over $\mathbb{C}$. So $G$ naturally acts on the $|\Omega|$-dimensional linear space $V$. The dimension of $[V,a]$ is $m-1$. The space $V$ is a direct sum of irreducible $G$-modules and there are at most $m$ of these (the trivial module and at most $m-1$ nontrivial ones). Each of the irreducible $G$-modules has $(m,|a|)$-bounded dimension by the Hartley-Isaacs Theorem B \cite{HI}. It follows that the dimension of $V$ is $m$-bounded, as required.
\end{proof}

The following lemma will be useful.

\begin{lemma}\label{55} Let $G$ be a group admitting a coprime automorphism $\alpha$ such that $G=[G,\al]$. Suppose that $G=NH$ is a product of an $\al$-invariant normal subgroup $N$ and an $\al$-invariant subgroup $H$. Assume that $[H,\al]$ is generated by $a_1,\dots,a_s$ while $[N,\al]$ is generated by $b_1,\dots,b_t$. Then $G=\langle a_1,\dots,a_s,b_1,\dots,b_t\rangle$.
\end{lemma}
\begin{proof} Since $G=N[H,\al]$, without loss of generality we can assume that $H=[H,\al]$. Thus, $H=\langle a_1,\ldots,a_s\rangle$. Hence the subgroup $\langle [N,\al]^H\rangle$ is  contained in $\langle a_1,\dots,a_s,b_1,\dots,b_t\rangle$. By Lemma \ref{20}(iii) the image of $N$ in the quotient group $G/\langle [N,\al]^H\rangle$ becomes central and therefore the image of $H$ becomes normal. Hence, $\langle [N,\al]^H\rangle H$ is normal in $G$. Obviously, $\al$ acts trivially on $G/\langle [N,\al]^H\rangle H$. Since $G=[G,\al]$, we conclude that $G=\langle [N,\al]^H\rangle H$ and the result follows. 
\end{proof}

In the sequel we will require the following well-known fact (see \cite[p.\ 271]{rose}).

\begin{lemma}\label{03} Let $N$ be a normal subgroup of a finite group $G$. Let $H$ be a minimal subgroup of $G$ such that $G=NH$. Then $H\cap N\leq\Phi(H)$.
\end{lemma}

\section{Theorems \ref{main2} and \ref{main3}}

In this section Theorem \ref{main2} and Theorem \ref{main3} will be proved. Naturally, the proof relies on the classification of simple groups. We will also extensively use standard facts about conjugacy classes and characters of $\mathrm{PGL}_2(q)$ (see \cite{Do}).  We first prove  results about $\mathrm{PSL}_2(q)$. 

\begin{lemma} \label{psl} Let $e\ge5$ be an odd positive integer, $q_0$ an odd prime power, and let $q=q_0^e$.  Let $\al$ be a field automorphism of $H:=\mathrm{PSL}_2(q)$ of order $e$ (necessarily with $C_H(\al)=\mathrm{PSL}_2(q_0))$. Let $C=\al^H$. Then $CC^{-1}=H$.
\end{lemma}

\begin{proof}  For convenience it is easier to work with $J=\mathrm{PGL}_2(q)$. Since
$J=HC_J(\al)$, this suffices (i.e. $\al^J=\al^H$). Set $M=J\langle\al\rangle$. 

Note that $J$ has $q+1$ conjugacy classes of order prime to $q$ and $1$ conjugacy class of elements whose order is not relatively prime to $q$. So $J$ has $q+2$ nontrivial
irreducible characters of dimensions $q-1, q,$ and $q+1$. Note that $\al$ leaves precisely
$q_0 + 2$ conjugacy classes invariant and so by Brauer's permutation lemma, the same
is true for irreducible characters.  

By the class equation formula, the lemma is equivalent to saying that
$$
\sum \phi(\al)\phi(\al^{-1})\phi(y)/\phi(1) \ne 0, \ \ \ \ \ \ \ \ \ \ \ \  \ \ (*)
$$
for every $y \in H$. Here the sum is taken over the irreducible characters $\phi$ of $M$.

Note that if $\phi$ is an irreducible character of $M$ and is not irreducible when restricted to $J$, then $\phi(\al)=0$.
So it suffices to consider the $\al$-invariant characters of $J$.  Each has precisely $e$ distinct extensions to $M$, which are the same up
to a twist by a linear character of $M/J$ and so give the same value in the sum above -- thus, to show that the class sum is nonzero, it suffices to pick one extension of each invariant character to $\langle J, \al \rangle$.   

By inspection of the character table for $\mathrm{PGL}_2(q)$, it follows that if $\phi$ is an irreducible character, $|\phi(y)| \le 2$ for each $y\in H$. If $\phi$ is $\al$-invariant, we claim that $|\phi(\al)| \le q_0 + 1$.

To see this, note that if $U$ is a Sylow $p$-subgroup of order $q$, then every nontrivial
character of $U$ occurs once and the trivial character with multiplicity at most $2$.
Thus, $\al$ permutes the nontrivial characters of $U$ fixing exactly $q_0 -1$ of them and so
the claim holds. 

If $\phi$ is not irreducible over $J$ (equivalently not $\al$-invariant), then $\phi(\al)=0$ and so does not contribute to the sum in $(*)$. If $\phi(1)=1$, then $\phi(y) =1$ as well (since $y\in H = [J,J]$) and so the contribution is $1$ for each. There are two such characters.  
 
There are $q_0$ irreducible characters not of degree $1$ that are $\al$-invariant. For any such character $\phi$, we have 
  $$ |\phi(\al) \phi(\al^{-1}) \phi(y)/\phi(1)| \le 2 (q_0+ 1)^2/(q-1). $$
 
Since there are $q_0$ such characters, the absolute value of the sum of these is at 
most $2q_0(q_0+1)^2/(q-1) < 2$. The last inequality follows since $q \ge q_0^5$.  This implies the claim.  
\end{proof}  
 
 Essentially the same proof yields:
 
\begin{lemma} \label{psleven}  Let $e\ge4$ be a positive integer, $q_0$ a power of $2$, and $q=q_0^e$. Let $\al$ be a field automorphism of $H:= \mathrm{PSL}_2(q)$ of order $e$ (so with centralizer $\mathrm{PSL}_2(q_0)$). Let $C=\al^H$. Then $CC^{-1}=H$.
\end{lemma}

\begin{proof} Note that since $q$ is even, $\mathrm{PSL}_2(q)= \mathrm{PGL}_2(q)$. In this case, there are $q$ conjugacy classes of elements of odd order and $1$ conjugacy class of involutions for a total of $q+1$ conjugacy classes. As in the previous result, we see that there are $q_0+1$ $\al$-invariant irreducible characters of $H$ of possible dimensions $q -1, q$ and $q+1$. The trivial character of $H$ is the only linear character. The estimates for the character values are the same as in the previous lemma and so each nontrivial character contributes at most $2(q_0+1)^2/(q-1)$. Since $e\ge4$ and there are $q_0$ characters to account for, the
absolute value of this sum is at most $2q_0(q_0+1)^2/(q-1) < 1$ (since $e\ge 4)$  unless possibly $q_0=2$ and $e \le 5$. In those two cases, one just computes the class sum directly to obtain the result. 
\end{proof} 

We are ready to prove Theorem \ref{main2}. For the reader's convenience we restate it here. 

\begin{theorem} \label{o(g)} Let $G$ be a finite group admitting a coprime automorphism $\al$ such that $g^{-1}g^\al$ has odd order for every $g\in G$. Then $[G,\al]\leq O(G)$.
\end{theorem}

\begin{proof} Assume that this is false and let $G$ be a counterexample of minimal order. Then $G=[G,\al]$ and $O(G)=1$. Let $M$ be a minimal $\al$-invariant normal subgroup of $G$. Since $G/M$ satisfies the hypothesis, by   induction $G/M$ has odd order. The subgroup $M$ is either elementary abelian or semisimple.

If $M$ is abelian, then $M$ is a 2-subgroup and so by hypotheses $M\leq C_G(\al)$. It follows from Lemma \ref{20}(iii) that $M\leq Z(G)$, which leads to a contradiction since $G/M$ has odd order and $O(G)=1$.

Hence, we can assume that $M$ is a direct product of isomorphic nonabelian simple groups and $\al$ transitively permutes the simple factors. Moreover because of minimality $G=M$. If $M$ is a product of more than one simple group, and if $S$ is a simple factor in which $g$ is an involution, observe that $g^{-1}g^\al$ has order two, a contradiction. 

So we are reduced to the case that $G$ is simple. It follows that $G$ is a group of Lie type and $\al$ is a field automorphism, say of coprime order $e$. 

If $G$ is a group of Lie type in characteristic $2$, then field automorphisms do not
centralize Sylow $2$-subgroups (but do normalize one) and so we have a contradiction again.

So $G$ is a group of Lie type in odd characteristic $p$, defined over a field of size
$q=q_0^e$ with all divisors of $e$ at least $5$ (the only case where there is a coprime
automorphism of order $3$ is for the Suzuki groups which are in characteristic $2$).

First suppose that $G$ has (twisted) Lie rank $1$. If $G=\mathrm{PSL}_2(q)$ we apply Lemma \ref{psl}.
If $G=\mathrm{PSU}_3(q)$ or $^{2}\mathrm{G}_2(q)$, we observe that there is an $\al$-invariant subgroup isomorphic to $\mathrm{PSL}_2(q)$ which is not centralized by $\al$ and so again the lemma applies.  

So we may assume that $G$ has rank at least $2$.  Note that $\al$ normalizes a Borel subgroup and, by the structure of field automorphisms, $\al$ normalizes each parabolic subgroup of $G$ containing $B$. We see that $\al$ normalizes a Levi subgroup $L$ (and so its derived subgroup). By choosing the parabolic subgroup to be minimal properly containing the Borel subgroup, we can assume that $L$ is of rank $1$. There may be a center but since the rank $1$ case reduces to $\mathrm{PSL}_2(q)$, the center will be a $2$-group. The result follows.  
\end{proof}  

We will now show that the coprimeness assumption is really necessary in Theorem \ref{main2}. First, quote a linear algebra result from \cite{gu1}.

\begin{theorem}\label{gu1} Let $k$ be an algebraically closed field and let $A, B \in M_n(k)$ be matrices such that $AB-BA$ has rank $1$.  Then $A$ and $B$ can be simultaneously triangularized.
\end{theorem}

This implies that if $F$ is any field and $x,y\in \mathrm{GL}_n(F)$ such that the group commutator $[x,y]$ is a transvection, then $\langle x,y\rangle$ has unipotent derived group, and in particular $\langle x,y\rangle$ is soluble.

To see this, we can assume that $F$ is algebraically closed and write $x^{-1}y^{-1}xy = I + A$, where $A$ is a nilpotent rank one matrix. Thus,  $xy -yx = yxA$ has rank $1$ and so $x, y$ are simultaneously triangular, whence the commutator subgroup $\langle x,y\rangle'$ is unipotent.

\begin{corollary}\label{gugu} Let $x\in GL(V)$ act on $V$ irreducibly. Then for any $y\in GL(V)$ the commutator $[x,y]$ is not a transvection.
\end{corollary}
\begin{proof} By way of contradiction suppose that there is $y\in GL(V)$ such that the commutator $[x,y]$ is a transvection. On the one hand, the subgroup $\langle x,y\rangle$ is irreducible because so is $x$. On the other hand, $\langle x,y\rangle'$ is unipotent and therefore $C_V(\langle x,y\rangle')\neq0$. This is a contradiction.
\end{proof}

The following example shows that Theorem \ref{main2} is no longer true if the assumption that the automorphism $\alpha$ is coprime is omitted.

Let $q = 2^a > 2$ and let $G = \mathrm{SL}_2(q)$. Let $x\in G$ be an element of order $q+1$. Then $[x,y]$ has odd order for all $y \in G$. This is because the only elements of even order in $G$ are transvections and, by Corollary \ref{gugu}, these are not commutators $[x,y]$.

We will now prove Theorem \ref{main3}. For the reader's convenience we restate it here. 
\begin{theorem} Let $G$ be a finite group admitting a coprime automorphism $\al$. If any pair of elements from $I_G(\al)$ generates a soluble subgroup, then $[G,\al]$ is soluble. If any pair of elements from $I_G(\al)$ generates a nilpotent subgroup, then $[G,\al]$ is nilpotent.
\end{theorem}

\begin{proof} Assume that any pair of elements from $I_G(\al)$ generates a soluble subgroup. We wish to prove that $[G,\al]$ is soluble. Assume that this is false and let $G$ be a counterexample of minimal order.   Arguing precisely as in the proof of Theorem \ref{o(g)}, we see that  either $G$ is simple or  $G$ is a product of $r >1$  copies of
a simple group $L$   and $\al$ permutes the factors transitively.   

Consider the second case.  By conjugating (in $\mathrm{Aut}(G)$), we may assume that 
$\al = (x,1, \ldots, 1) \rho$ where $x$ is an automorphism of $L$ and $\rho$ permutes the coordinates of $G$.   
It is clear that an element of  $I_G(\al)$ can have an arbitrary first coordinate in $L$
and since $L$ can be generated by $2$ elements, the result holds in this case. 

Assume that $G$ is simple. 
As in the proof of Theorem \ref{o(g)}, by minimality 
it follows that $G=\mathrm{PSL}_2(q)$ or $\mathrm{Sz}(q)$.

Suppose first that $G=\mathrm{PSL}_2(q)$. Lemmas \ref{psl} and \ref{psleven} imply that every element of $G$ is conjugate to some element in $I_G(\al)$.  In particular, 
there are elements of order $(q \pm 1)$ (if $q$ is even) or $(q \pm 1)/2$ (if $q$ is odd).   Since $q \ge 32$, there are no proper 
subgroups containing elements of both orders.  

Therefore $G=\mathrm{Sz}(q)$ and proper $\al$-invariant subgroups of $G$ are either soluble or contained in $C_G(\al)$. Choose two $\al$-invariant cyclic Hall subgroups $J$ and $K$ of order $q+\sqrt{2q}+1$ and $q-\sqrt{2q}+1$, respectively. It is straightforward that $\langle[J,\al],[K,\al]\rangle$ is a nontrivial $\al$-invariant subgroup generated by two elements from $I_G(\al)$. Hence, $\langle[J,\al],[K,\al]\rangle$ is soluble. The subgroup structure of the Suzuki groups is given in \cite[p. 117]{Wilson}. We see that no proper subgroup of $G$ contains $\langle[J,\al],[K,\al]\rangle$, a contradiction. 

Thus, $[G,\al]$ is soluble, as claimed. We will now show that if any pair of elements from $I_G(\al)$ generates a nilpotent subgroup, then $[G,\al]$ is nilpotent. 

Again, let $G$ be a counterexample of minimal order. Then $G=[G,\al]=NH$, where $N$ is an $\al$-invariant elementary abelian normal $p$-subgroup and $H$ is an $\al$-invariant nilpotent $p'$-subgroup such that $H=[H,\al]$. By Lemma \ref{55}, $G=\langle I_N(\al),I_H(\al)\rangle$. Since any pair of elements from $I_G(\al)$ generates a nilpotent subgroup and since $(|N|,|H|)=1$, we deduce that $I_N(\al)$ centralizes $I_H(\al)$. Taking into account that $N$ is abelian deduce that $I_N(\al)\leq Z(G)$. It follows that $G/Z(G)$ is nilpotent and this completes the proof.
\end{proof}

\section{Proof of Theorem \ref{main}}

We are ready to embark on the proof of Theorem \ref{main}. It will be convenient to deal separately with the case where $G$ is nilpotent.

\subsection{The case of nilpotent groups}

As usual, we write $Z_i(H)$ and $\gamma_i(H)$ for the $i$th term of the upper and lower central series of a group $H$, respectively.

\begin{lemma}\label{22}
Let $p$ be a prime and $G$ a group admitting a coprime automorphism $\al$ such that $G=[G,\al]$. Let $M$ be an $\al$-invariant normal $p$-subgroup of $G$ and assume that $|I_M(\al)|=p^m$ for some nonnegative integer $m$. Then $M\leq Z_{2m+1}(O_p(G))$. 
\end{lemma} 
\begin{proof} If $m=0$, then the result is immediate from Lemma \ref{20}(iii), so assume that $m\geq 1$ and use induction on $m$.

Let $K=O_p(G)$ and $N=M\cap Z_2(K)$. If $N\not\leq Z(K)$, then Lemma \ref{20}(iii) implies that $I_N(\al)\neq 1$, in which case we have $|I_{M/N}(\al)|<|I_M(\al)|=p^m$. By induction $M/N \leq Z_{2m-1}(K/N)$, whence $M\leq Z_{2m+1}(K)$. If $N\leq Z(K)$, then it turns out that  $M\cap Z(K)=M\cap Z_i(K)$ for any $i\geq 2$ and  so, obviously,  $M\leq Z(K)$. This concludes the proof.
\end{proof}

The following result is well known (see for example \cite[Lemma 2.2]{shumy98}). It will be useful later on.
\begin{lemma}\label{24}
Let $G$ be a group of prime exponent $p$ and rank $r_0$. Then there exists a number  $s=s(r_0)$, depending only on $r_0$, such that $|G|\leq p^s$.
\end{lemma}

Throughout this subsection, unless stated otherwise, $G$ is a $p$-group admitting a coprime automorphism $\al$ such that $G=[G,\al]$ and any subgroup generated by a subset of $I_G(\al)$ can be generated by at most $r$ elements.

\begin{lemma}\label{23}  Suppose that $G$ is of prime exponent $p$. There exists a number $l=l(r)$, depending on $r$ only, such that the rank $r(G)$ of $G$ is at most $l$.
\end{lemma}

\begin{proof}
Let $C$ be Thompson's critical subgroup of $G$ (see \cite[Theorem 5.3.11]{gore}, and set $A=Z(C)$. Observe that $[A,\al]$ is an $r$-generated abelian subgroup of exponent $p$ and so the order of $[A,\al]$ is at most $p^r$. By Lemma \ref{22} $A$ is contained in $Z_{2r+1}(G)$. Since $[G,C]$ is contained in $A$, we conclude that $C$ is contained in $Z_{2r+2}(G)$. Recall that $\gamma_{2r+2}(G)$ commutes with $Z_{2r+2}(G)$ and so in particular $\gamma_{2r+2} (G)$ centralizes $C$. Again by Thompson's theorem, $C_G(C)=A$. Thus $\gamma_{2r+2}(G)$ is contained in $A$, that is, the quotient group $G/A$ is nilpotent of class $2r+1$. We deduce that $G$ has $r$-bounded nilpotency class. Since $G=[G,\al]$ is $r$-generated by hypothesis, it follows that the rank $r(G)$ of $G$ is $r$-bounded, as desired. 
\end{proof}

We will require the concept of powerful $p$-groups. These were introduced by Lubotzky and Mann in \cite{LM}: a finite $p$-group $H$ is powerful if and only if $H^p\leq[H,H]$ for $p\neq2$  (or $H^4\leq[H,H]$ for $p = 2$). The reader can consult books \cite{DDM} or \cite{khukhu2} for more information on these groups.

\begin{lemma}\label{powerful}There exists a number $\lambda=\lambda(r)$, depending only on $r$, such that $\gamma_{2\lambda+1}(G)$ is powerful.
\end{lemma}
\begin{proof}
Let $s(r_0)$ be as in Lemma \ref{24} and let $l(r)$ be as in Lemma \ref{23}. Take $N=\gamma_{2\lambda+1}(G)$, where $\lambda=s(l(r))$. In order to show that $N'\leq N^p$, we assume that $N$ is of exponent $p$ and prove that $N$ is abelian. 

Note that the subgroup $[N,\al]$ is of exponent $p$. By Lemma \ref{23} the rank of $[N,\al]$ is at most $l(r)$.  It  follows from Lemma \ref{24} that $|[N,\al]|\leq p^{s(l(r))}=p^\lambda$. Now Lemma \ref{22} yields $N\leq Z_{2\lambda+1}(G)$. Since $[\gamma_i(G),Z_i(G)]=1$ for any positive integer $i$, we conclude that $N$ is abelian, as required. 
\end{proof}

\begin{lemma}\label{25}
For any $i\geq 1$, there exists a number $m_i=m_i(i,r)$, depending only on $i$ and $r$, such that $\gamma_i(G)$ is an $m_i$-generated group.\end{lemma}
\begin{proof}
Let $N=\gamma_i(G)$. We can pass to the quotient $G/\Phi(N)$ and assume that $N$ is elementary abelian. It follows that $|I_N(\al)|\leq p^r$. Thus, by Lemma \ref{22}, we have $N\leq Z_{2r+1}(G)$ and deduce that $G$ has nilpotency class bounded  only in terms of $i$ and $r$. Since $G=[G,\al]$ is $r$-generated, we conclude that $r(G)$ is $(i,r)$-bounded as well. Therefore $N$ is $m_i$-generated for some $(i,r)$-bounded number $m_i$. This concludes the proof.
\end{proof}

The next proposition shows that Theorem \ref{main} is valid in the case where $G$ is a $p$-group.
\begin{proposition}\label{pcase}
The rank of $G$ is $r$-bounded.
\end{proposition}
\begin{proof}
Let  $s(r_0)$ be as in Lemma \ref{24} and $l(r)$ as in Lemma \ref{23}. Take $N=\gamma_{2\lambda+1}(G)$, where $\lambda=\lambda(r)=s(l(r))$. Let $d$ be the minimal number such that $N$ is $d$-generated. Lemma \ref{25} tells us that $d$ is an $r$-bounded integer. Moreover, by Lemma \ref{powerful} $N$ is powerful. It follows from \cite[Theorem 2.9]{DDM} that $r(N)\leq d$, and so the rank of $N$ is $r$-bounded.  Since  the nilpotency class of $G/N$ is $r$-bounded (recall that $\lambda$ depends only on $r$) and $G=[G,\al]$ is $r$-generated, we conclude that  $r(G/N)$ is $r$-bounded as well. Note that $r(G)\leq r(G/N)+r(N)$ and the result follows. 
\end{proof}

\begin{corollary}\label{nilpotent} Assume the hypotheses of Theorem \ref{main} and let $G$ be nilpotent. Then the rank of $[G,\al]$ is $r$-bounded.
\end{corollary}
\begin{proof} The rank of $[G,\al]$ is equal to the rank of $[P,\al]$, where $P$ is some Sylow $p$-subgroup of $G$, and the result easily follows from Proposition \ref{pcase}.
\end{proof}

\subsection{The case of soluble groups}

As usual, we denote by $F(G)$ the Fitting subgroup of a group $G$. Write $F_0(G)=1, F_1(G)=F(G)$ and let $F_{i+1}(G)$ be the inverse image of $F(G/F_i(G))$. If $G$ is soluble, then the least number $h$ such that $F_h(G)=G$ is called the Fitting height of $G$.

The purpose of this subsection is to show that if under the hypotheses of Theorem \ref{main} the group $G$ is soluble, then $h([G,\al])$ is $(e,r)$-bounded and moreover $[G,\al]$ can be generated by $(e,r)$-boundedly many elements from $I_G(\al)$. One key step consists in showing  that there exists an $(e,r)$-bounded number $f$ such that the $f$th term of the derived series of $[G,\alpha]$ is nilpotent. For this we will require the following result which is an immediate corollary of Hartley-Isaacs Theorem B in \cite{HI}.

\begin{proposition}\label{cristina} Let $H$ be a finite soluble group admitting a coprime automorphism $\al$ of order $e$ such that $H=[H,\al]$. Let $\it{k}$ be any field with characteristic prime to $e$, and $V$ a simple $\it{k}H\langle\al\rangle$-module. Suppose that $\dim [V,\al]=r$. There exists an $(e,r)$-bounded number $\delta=\delta(e,r)$ such that $\dim V\leq \delta$.
\end{proposition}

In the proof of the next proposition we will use the well-known theorem of Zassenhaus (see \cite[Satz 7]{Zas} or \cite[Theorem 3.23]{rob}) stating that for any $n\geq1$ there exists a number $j=j(n)$, depending only on $n$, such that, whenever $\it{k}$ is a field, the derived length of any soluble subgroup of $\mathrm{GL}_n({k})$ is at most $j$.

\begin{proposition}\label{32}
Assume the hypotheses of Theorem \ref{main}. Suppose that $G$ is soluble and $G=[G,\al]$. There exists a number $f=f(e,r)$, depending only on $e$ and $r$, such that the $f$th term  $G^{(f)}$ of the derived series of $G$ is nilpotent. 
\end{proposition}

\begin{proof} Let $\delta=\delta(e,r)$ be as in Proposition \ref{cristina} and $f=j(\delta)$ be the number given by the Zassenhaus theorem.

Suppose that the proposition is false and let $G$ be a group of minimal possible order such that the hypotheses hold while $G^{(f)}$ is not nilpotent. Then $G$ has a unique minimal $\al$-invariant normal subgroup $M$. Indeed, suppose that $G$ has two minimal $\al$-invariant normal subgroups, say $M_1$ and $M_2$.  Then $M_1\cap M_2=1$. Since $|G/M_1|<|G|$, the minimality of $G$ implies that $(G/M_1)^{(f)}$ is nilpotent. By a symmetric argument $(G/M_2)^{(f)}$ is nilpotent too. This yields a contradiction since $G^{(f)}$  can  be embedded into a nilpotent subgroup of $G/M_1\times G/M_2$.

We claim that $M=C_G(M)$. Since $M$ is a $p$-subgroup for some prime $p$ and because of the uniqueness of $M$, the Fitting subgroup $F=F(G)$ is a $p$-subgroup too. If $\Phi(F)$  is nontrivial, then  we immediately get a contradiction because $F(G/\Phi(F))=F/\Phi(F)$ and, again by the minimality of $G$, we know that  $(G/\Phi(F))^{(f)}$ is nilpotent, so  in particular $G^{(f)} \leq F$.  

So assume that $\Phi(F)=1$ and thus $F$ is elementary abelian. If $M=F$, then $M=C_G(M)$ since the Fitting subgroup of a soluble group contains its own centralizer (see, for example, \cite[Theorem 1.3, Chap. 6]{gore}). Thus we can assume that $M<F$. By hypotheses, on one hand, we know that  $G^{(f)}\leq F_2(G)$ and, on the other hand, $(G/M)^{(f)}$ is nilpotent (again by the minimality of $G$). Now let $T$ be an $\al$-invariant Hall $p'$-subgroup of $G^{(f)}$. It follows that both $FT$ and $MT$ are $\al$-invariant normal subgroups of $G$. Indeed, $FT/F$ is normal in $G/F$, since  $(G/F)^{(f)}$ is nilpotent  and, similarly, $MT/M$ is normal in $G/M$ since $(G/M)^{(f)}$ is nilpotent as well.

Suppose that $C_F(T)\neq1$. Note that $C_F(T)=Z(FT)$ since $F$ is abelian. Thus $C_F(T)$  is an $\al$-invariant normal subgroup of $G$ because  $FT$ is normal and $\al$-invariant. Hence  $M\leq C_F(T)$. This implies that  $T$ centralizes $M$ and so $MT=T\times M$. Recall that $T\leq F_2(G)$ and $T\cap F=1$. It follows that $T$ is nilpotent.  Then $T\times M$ is normal nilpotent and $T\leq F$, a contradiction. 

Thus, $C_F(T)=1$. On the other hand, we see that $[F,T]\leq M$, since the nilpotent $p'$-subgroup $MT/M$ and the $p$-subgroup $F/M$  are both contained in $F(G/M)$ and therefore commute. Now we have $M<F$ and $F=[F,T]\times C_F(T)$, so it should be $C_F(T)\neq1$, a contradiction. Thus $M=C_G(M)$, as claimed above.

Therefore $G/M$ acts faithfully and irreducibly on $M$. Moreover $[M,\al]$ is $r$-generated and elementary abelian, so $|[M,\al]|\leq p^r$. We view $M$ as a $G/M\langle \al\rangle$-module over the field with $p$ elements. Observe that $p$ does not divide $e$, since $\al$ is a coprime automorphism. By Proposition \ref{cristina} we have  $\dim(M)\leq\delta(e,r)$. Applying the theorem of Zassenhaus conclude that the derived length of $G/M$ is at most $f=f(\delta(e,r))$. Then $G^{(f)}\leq F$, which concludes the proof. 
\end{proof}

As a by-product of the previous result we deduce that the Fitting height of $G$ is $(e,r)$-bounded. 
\begin{corollary}\label{cor}
Under the hypothesis of Proposition \ref{32} the Fitting height $h(G)$ is $(e,r)$-bounded.
\end{corollary}
\begin{proof} 
By Proposition \ref{32} we know that $G^{(f)}$ is nilpotent for some $(e,r)$-bounded number $f$. The result follows since $h(G)\leq f+1$. 
\end{proof}
\begin{proposition}\label{soluble}
Under the hypothesis of Proposition \ref{32} the group $G$ is generated by $(e,r)$-boundedly many elements from $I_G(\al)$.
\end{proposition}
\begin{proof} If $G$ is a $p$-group, then the claim follows from the Burnside Basis Theorem since $G=[G,\al]$ is $r$-generated. In the case where $G$ is nilpotent, we have $G=[P_1,\al]\times\cdots\times[P_s,\al]$, where $\{P_1,\ldots,P_s\}$ are the Sylow subgroups of $G$. So it   follows from the case of $p$-groups that $G$ is generated by $r$ elements from $I_G(\al)$.

Assume that $G$ is not nilpotent. Let $h=h(G)\geq2$. Since we know from Corollary \ref{cor} that $h$ is $(e,r)$-bounded, we argue by induction on $h$. Let $F=F(G)$. By induction there are $(e,r)$-boundedly many elements $a_1,\dots,a_d\in I_G(\al)$ such that $G=F\langle a_1,\dots,a_d\rangle$. We can choose $a_1,\dots,a_d$ in such a way that the subgroup $H=\langle a_1,\dots,a_d\rangle$ is $\al$-invariant. We have seen in the previous paragraph that $[F,\al]$ can be generated by at most $r$ elements from $I_F(\al)$. Thus, Lemma \ref{55} tells us that $G$ can be generated by $d+r$ elements from $I_G(\al)$.
\end{proof}

\subsection{The general case}

Let $G$ be a finite group admitting a coprime automorphism $\al$ of order $e$ such that any subset of $I_G(\al)$ generates an $r$-generator subgroup. We want to prove that $[G,\al]$ has $(e,r)$-bounded rank. Thus, throughout the remaining part of the paper we assume that $G=[G,\al]$.

\begin{lemma}\label{1} If $G$ is simple, then the rank of $G$ is $r$-bounded.
\end{lemma}
\begin{proof} This is immediate from Lemma \ref{0000} (2). 
\end{proof}

\begin{lemma}\label{2} Suppose that $G$ is semisimple and $\al$ transitively permutes the simple factors. Then the rank of $G$ is $(e,r)$-bounded.
\end{lemma}
\begin{proof} Write $G=S_1\times\dots\times S_k$. Since the case $k=1$ was considered in Lemma \ref{1}, we assume that $k\geq2$. Here $k$ is a divisor of $e$ and so it is sufficient to show that the rank of $S_1$ is at most $r$. Suppose that this is not the case and choose a subgroup $H\leq S_1$ which needs  at least  $r+1$ generators. Consider the subgroup $K\leq S_1\times S_1^\al$ generated by all elements of the form $x^{-1}x^\al$, where $x\in H$. On the one hand, $K$ is generated by a subset of $I_G(\al)$ and so it can be generated by $r$ elements. On the other hand, $H$ is a homomorphic image of $K$ and so we have a contradiction with the fact that $H$ cannot be generated with $r$ elements.
\end{proof}

\begin{lemma}\label{3} Suppose that $G$ is semisimple. Then the rank of $G$ is $(e,r)$-bounded.
\end{lemma}
\begin{proof} Since $G=[G,\al]$, it follows that $G=G_1\times\dots\times G_m$, where each factor $G_i$ is either simple such that $G_i=[G_i,\al]$ or a direct product of more than one simple groups which are transitively permuted by $\al$. We already know from the two previous lemmas that the rank of $G_i$ is $(e,r)$-bounded so it remains to show that the number $m$ of such factors is $(e,r)$-bounded too. In view of Theorem \ref{main2} each subgroup $G_i$ has an element $g_i$ such that $x_i=g_i^{-1}g_i^\al$ has even order. The abelian subgroup $\langle x_1,\dots,x_m\rangle$ has Sylow 2-subgroup of rank $m$ and so it cannot be generated with less than $m$ elements. Hence, $m\leq r$.
\end{proof}
Write $G_0=G\langle\al\rangle$.
\begin{lemma}\label{444} Let $N$ be an $\al$-invariant normal subgroup of $G$ and assume that $N=S_1\times\dots\times S_l$ is a direct product of nonabelian simple factors $S_i$. Then both $l$ and the rank of $N$ are $(e,r)$-bounded.
\end{lemma}
\begin{proof}  In view of Lemma \ref{3} the rank of $[N,\al]$ is $(e,r)$-bounded. Since all factors $S_i$ have even order and since the rank of the Sylow $2$-subgroup of $[N,\al]$ is $(e,r)$-bounded, it follows that only $(e,r)$-boundedly many, say $m$, of the subgroups $S_1,\dots,S_l$ are not contained in $C_G(\al)$. On the other hand, because of Lemma \ref{20}(iii) no nontrivial normal subgroup of $G$ can be contained in $C_N(\al)$. Thus, every simple factor in the list $S_1,\dots,S_l$ is conjugate in $G$ with a factor which is not centralized by $\al$ and so by Lemma \ref{3} each $S_i$ has $(e,r)$-bounded rank. Hence, we only need to show that $l$ is  $(e,r)$-bounded.

The group $G_0$ naturally acts on the set $\{S_1,\dots,S_l\}$ by conjugation. The above argument shows that there are at most $m$ $G_0$-orbits in this action. It is sufficient to show that each $G_0$-orbit has $(e,r)$-bounded length. Let $K$ be the kernel of the action, that is, the intersection of normalizers of $S_i$. It is straightforward from Lemma \ref{02} that the index of $K$ in $G_0$ is $m$-bounded. Since the length of each $G_0$-orbit is at most the index $[G_0:K]$, the result follows. 
\end{proof}

\begin{lemma}\label{44} Suppose that $G$ has an $\al$-invariant subgroup $K$ of index $i$ such that $[K,\al]$ is of rank $s$. Then the rank of $G$ is $(i,s)$-bounded.
\end{lemma}
\begin{proof} We can assume that $K$ is normal in $G$. Since $[K,\al]$ is normal in $K$, it follows that the index of the normalizer of $[K,\al]$ in $G$ is a divisor of $i$. Using the fact that the rank of $[K,\al]$ is $s$ we conclude that the rank of the normal closure $K_1$ of $[K,\al]$ is $(i,s)$-bounded. In view of Lemma \ref{20}(iii) the quotient $K/K_1$ is central in $G/K_1$. Hence, by Schur's Theorem  \cite[Theorem 4.12]{rob}, the image in $G/K_1$ of the commutator subgroup $G'$ has $i$-bounded order. Therefore we can pass to the quotient $G/K_1G'$ and assume that $G$ is abelian. In this case the lemma is obvious. \end{proof}

We will now establish several lemmas about generation of $G$ by elements from $I_G(\al)$. Recall that Proposition \ref{soluble} tells us that if $G$ is soluble, then $G$ can be generated by an $(e,r)$-bounded number of elements from $I_G(\al)$.

\begin{lemma}\label{5554} If $G$ is semisimple, then $G$ can be generated by an $(e,r)$-bounded number of elements from $I_G(\al)$.
\end{lemma}
\begin{proof} Let $G=S_1\times\dots\times S_l$ where the factors $S_i$ are simple. The automorphism $\al$ permutes the simple factors and the proof of Lemma \ref{3} shows that there are at most $r$ orbits under this action. Therefore without loss of generality we assume that $\al$ transitively permutes the factors $S_i$ and so $l$ is a divisor of $e$. If $G$ is simple, then by Lemma \ref{1111} $G$ is generated by two nilpotent subgroups $P_1$ and $P_2$ such that $[P_1,\al]=P_1$ and $[P_2,\al]=P_2$. Each of the subgroups $P_i$ is generated by at most $r$ elements from $I_G(\al)$ and so $G$ is generated by at most $2r$ such elements. We will therefore assume that $l\geq2$.

We will use the fact each nonabelian simple group can be generated by two elements. Let $a,b$ generate $S_1$.

Set $$x_1=a^{-1}a^\al,\ x_2=b^{-1}b^\al \text{ and } x_3=ab((ab)^{-1})^\al.$$ Note that all $x_i$ belong to $I_G(\al)$. Let $K$ be the minimal $\al$-invariant subgroup of $G$ containing $x_1,x_2$, and $x_3$. Obviously $K$ is generated by at most $3e$ elements from $I_G(\al)$. Observe that $1\neq x_1x_2x_3=[a,b]\in S_1\cap K$. Evidently, the projection of $K$ to $S_1$ is the whole group $S_1$, that is, $K$ is a subdirect product of the factors $S_i$. We deduce that the conjugacy class $[a,b]^K$ generates $S_1$ and so $S_1$ is contained in $K$. Since $K$ is $\al$-invariant we are forced to conclude that $K=G$ and the result follows.
\end{proof}

\begin{lemma}\label{5} Suppose that $G$ is semisimple-by-soluble. Then $G$ is generated by an $(e,r)$-bounded number of elements from $I_G(\al)$.
\end{lemma}
\begin{proof} Let $N$ be an $\al$-invariant normal semisimple subgroup of $G$ such that $G/N$ is soluble. Choose a minimal subgroup $H_0$ of $G_0$ such that $G_0=NH_0$. Without loss of generality we can assume that $\al\in H_0$. Set $H=H_0\cap G$ and note that $H=[H,\al]$. Note that $H$ is soluble by Lemma \ref{03}. We have $G=NH$ and we know from Proposition \ref{soluble} and Lemma \ref{5554} that both $H$ and $[N,\al]$ can be generated by an $(e,r)$-bounded number of elements from $I_G(\al)$. The result follows from Lemma \ref{55}.
\end{proof}

\begin{lemma}\label{7} Assume that $G_0/\Phi(G_0)$ is semisimple-by-soluble. Then $G$ is generated by an $(e,r)$-bounded number of elements from $I_G(\al)$.
\end{lemma}
\begin{proof} Let $\overline{G_0}=G_0/\Phi(G_0)$ and denote by $\overline{G}$ the image of $G$ in $\overline{G_0}$. By Lemma \ref{5} $\overline{G}$ is generated by $(e,r)$-boundedly many elements from $I_{\overline{G}}(\al)$, say $\overline{x_1},\ldots, \overline{x_s}$ and so $\overline{G_0}=\langle \overline{\al}, \overline{x_1},\ldots, \overline{x_s}\rangle$. Hence $G_0=\langle \al, x_1,\ldots, x_s\rangle$, where $x_1,\ldots,x_s\in I_G(\al)$. Thus $G$ is generated by the $\al$-orbits of $x_1,\ldots,x_s$ and the result follows.
\end{proof}
In what follows $S(K)$ denotes the soluble radical of a group $K$.
\begin{lemma}\label{8} Suppose that $G$ is soluble-by-semisimple-by-soluble. Then $G$ is generated by an $(e,r)$-bounded number of  elements from $I_G(\al)$.
\end{lemma}

\begin{proof} Let $S=S(G)$. Let $H_0$ be a minimal subgroup of $G_0$ such that $G_0=SH_0$. Again, without loss of generality we can assume that $\al\in H_0$. Since $H_0\cap S\leq\Phi(H_0)$, Lemma \ref{7} shows that $H=H_0\cap G$ is generated by an $(e,r)$-bounded number of  elements from $I_G(\al)$. Note that by Proposition \ref{soluble} also $[S,\al]$ is generated by an $(e,r)$-bounded number of  elements from $I_G(\al)$. The result follows from Lemma \ref{55}.
\end{proof} 
\begin{lemma}\label{4} Assume that $S(G)=1$. Then $G$ has an $\al$-invariant semisimple-by-soluble normal subgroup of $(e,r)$-bounded index.
\end{lemma}
\begin{proof} Let $N=S_1\times\dots\times S_l$ be the socle of $G$. Here $S_1,\dots,S_l$ are the subnormal simple subgroups. In view of Lemma \ref{444} $l$ is $(e,r)$-bounded. The group $G_0$ naturally acts on the set $\{S_1,\dots,S_l\}$ by conjugation. Let $K$ be the kernel of this action. By \cite[Lemma 2.1]{ijm2015}, the quotient $K/N$ is soluble. Therefore $K$ is a semisimple-by-soluble normal subgroup of $(e,r)$-bounded index.
\end{proof}

\begin{lemma}\label{9} The group $G_0$ has a soluble-by-semisimple-by-soluble normal subgroup of $(e,r)$-bounded index.
\end{lemma}
\begin{proof} This is immediate from Lemma \ref{4}.
\end{proof}

In view of Lemma \ref{44} it is sufficient to prove Theorem \ref{main} in the case where $G$ is soluble-by-semisimple-by-soluble. We already know that in this case $G$ is generated by $(e,r)$-boundedly many elements from $I_G(\al)$.

\begin{lemma}\label{339} Assume that $G$ is soluble-by-semisimple-by-soluble. Let $N$ be an $\al$-invariant abelian normal subgroup of $G$. Then $[N,G]$ has $(e,r)$-bounded rank.
\end{lemma}
\begin{proof} By Lemma \ref{8} we know that $G$ is generated by $(e,r)$-boundedly many elements from $I_G(\al)$, say $b_1,\ldots,b_t$. Note that $[N,G]=[N,b_1]\ldots[N,b_t]$. So it is sufficient to bound the rank of $[N,b_i]$, for each $b_i\in \{b_1,\ldots,b_t\}$. By Lemma  \ref{20}(i)  we have $N=C_N(\al)\times[N,\al]$.  Take any $b\in  \{b_1,\ldots,b_t\}$ and choose $a\in G$ such that $b=a^{-1}a^{\al}$. Set $N_0=C_N(\al)\cap C_N(\al)^{a^{-1}}$. Since $[N,\al]$ has rank $r$ by hypothesis, we have $r(N/N_0)\leq 2r$. We claim that  $N_0\leq C_G(b)$. Indeed, choose $x\in C_N(\al)$ such that $x^{a^{-1}}\in C_N(\al)$. Then, we have $x^{a^{-1}}=(x^{a^{-1}})^\al$ and so $x$ commutes with $b=a^{-1}a^{\al}$ as claimed. Choose now elements $x_1,\ldots, x_{2r}$  that generate $N$ modulo $N_0$. By using linearity in $N$ and the fact that $N_0$ centralizes $b$, we deduce that $[N,b]$ is generated by $[x_1,b], \ldots,[x_{2r},b]$. Hence the result.
\end{proof}

\begin{proof}[Proof of Theorem \ref{main}] Recall that $G$ is a finite group admitting a coprime automorphism $\al$ of order $e$ such that any subgroup generated by a subset of $I_G(\al)$ can be generated by $r$ elements. We wish to prove that $[G,\al]$ has $(e,r)$-bounded rank. Without loss of generality we assume that $G=[G,\al]$.

As noted above, the combination of Lemma \ref{44} and Lemma \ref{9} ensures that it is sufficient to prove the result in the case where $G$ is soluble-by-semisimple-by-soluble. Hence, we assume that $G$ has a characteristic series $$1\leq S\leq T\leq G$$ such that $S$ and $G/T$ are soluble while $T/S$ is semisimple. Corollary \ref{cor} shows that the Fitting height of $G/T$ and $[S,\al]$ is $(e,r)$-bounded. Note that $[S,\al]$ is subnormal in $G$. Therefore the Fitting height of the normal closure $\langle [S,\al]^G\rangle$ equals that of $[S,\al]$. In view of Lemma \ref{20}(iii) the quotient $S/\langle [S,\al]^G\rangle$ is central in $G/\langle [S,\al]^G\rangle$. Therefore $G$ has a characteristic series of $(e,r)$-bounded length, say $l=l(e,r)$, all of whose factors are either semisimple or nilpotent. Moreover, there is at most one semisimple factor in the series and, by Lemma \ref{3}, it is of $(e,r)$-bounded rank. We will prove the theorem by induction on $l$.

If $l=1$, then $G$ is either semisimple or nilpotent. In the former case the result follows from Lemma \ref{3} and in the latter one from Corollary \ref{nilpotent}. Therefore we assume that $l\geq 2$. Let $N$ be the last term of the series. By induction $G/N$ has $(e,r)$-bounded rank. 

If $N\leq Z(G)$, then the rank of $G/Z(G)$ is bounded and a theorem of Lubotzky and Mann \cite{LM} guarantees that $G'$ has $(e,r)$-bounded rank (see also \cite{KS}). Thus  we can pass to $G/G'$ and simply assume that $G$ is abelian, whence the result is immediate. We therefore assume that $N$ is not central in $G$. If $N$ is semisimple, we have nothing to prove since the semisimple quotient of the series has $(e,r)$-bounded rank. Hence, we assume that $N$ is nilpotent. In this case the rank of $N$ is equal to the  rank of some Sylow $p$-subgroup $P$ of $N$. Thus, passing to $G/O_{p'}(N)$ without loss of generality, we can assume that $N=P$.  

We note that $P$ has an $(e,r)$-bounded number of generators. Indeed, pass to the quotient $G/\Phi(P)$ and assume that $P$ is elementary abelian. By Lemma \ref{339} $[P,G]$ has $(e,r)$-bounded rank and by the above this is also true for $G/[P,G]$. Hence, $P$ has $(e,r)$-boundedly many generators as well. 

Next, we claim that for any $i\geq 2$ there exists a number $m_i=m_i(i,e,r)$, depending only on $i,e$ and $r$, such that $V=\gamma_i(P)$ has $m_i$-bounded number of generators. We can pass to the quotient $G/\Phi(V)$ and assume that $V$ is elementary abelian. Now $[V,\al]$ is an elementary abelian $r$-generated group, so $|[V,\al]|\leq p^r$. Thus, \ by Lemma \ref{22},  we have $V\leq  Z_{2r+1}(O_p(G))$ so, in particular, $V\leq Z_{2r+1}(P)$ and deduce that the nilpotency class of $P/\Phi(V)$ is bounded in terms of $i$ and $r$ only. Since $P$ has an $(e,r)$-bounded number of generators, we conclude that $r(P/\Phi(V))$ is $(i,e,r)$-bounded as well. Therefore $V$ is $m_i$-generated for some $(i,e,r)$-bounded number $m_i$, as claimed.

Let $s(r_0)$ be as in Lemma \ref{24} and let $l(r)$ be as in Lemma \ref{23}. Take $M=\gamma_{2\lambda+1}(P)$, where $\lambda=s(l(r))$.  We want to prove that $M$ is powerful. In order to show that $M'\leq M^p$, we assume that $M$ is of exponent $p$ and prove that $M$ is abelian.  Note that the subgroup $[M,\al]$ is of exponent $p$. By Lemma \ref{23}  the rank of $[M,\al]$ is at most $l(r)$. It follows from Lemma \ref{24} that $|I_M(\al)|\leq p^{s(l(r))}=p^\lambda$. Now Lemma \ref{22} yields that $M\leq Z_{2\lambda+1}(P)$ . Since $[\gamma_i(P),Z_i(P)]=1$, for any positive integer $i$, we conclude that $M$ is abelian, as required. 

Let now $d_0$ be the minimal number such that $M$ is $d_0$-generated. It was shown above that $d_0$ is an $(e,r)$-bounded integer. Since $M$ is powerful, it follows from \cite[Theorem 2.9]{DDM} that $r(M)\leq d_0$, and so the rank of $M$ is $(e,r)$-bounded.  Since  the nilpotency class of $P/M$ is $(e,r)$-bounded and $P$ has $(e,r)$-boundedly many generators, we deduce that  $r(P/M)$ is $(e,r)$-bounded as well. Now $r(P)\leq r(P/M)+r(M)$ and the result follows. This concludes the proof.
\end{proof}

\end{document}